\documentclass[10pt,reqno]{amsart}

\usepackage{amsmath, amssymb, amsthm}
\usepackage{geometry}
\newtheorem{theorem}{Theorem}
\newtheorem{lemma}[theorem]{Lemma}
\newtheorem{proposition}[theorem]{Proposition}

\theoremstyle{definition}

\begin{document}
	\title[Asymptotics of Second order Linear ODEs]
	{Characterisation of Convergence, Boundedness and Unboundedness in Solutions of Second Order Linear Differential Equations}
	
	\author{John A. D. Appleby}
	\address{School of Mathematical
		Sciences, Dublin City University, Glasnevin, Dublin 9, Ireland}
	\email{john.appleby@dcu.ie} 
	
	\author{Subham Pal}
	\address{School of Mathematical
		Sciences, Dublin City University, Glasnevin, Dublin 9, Ireland}
	\email{subham.pal2@mail.dcu.ie}
	
	\date{7 March 2026}
	
	\begin{abstract}
		This paper develops a characterisation of when solutions of forced second order linear differential equations converge to the zero solution of the asymptotically stable and unforced second order equation, or when the solution is bounded, but not convergent, or is unbounded. We see thereby that forcing terms can exhibit unbounded and high--frequency oscillation, and yet the solution may still tend to zero, even though the first and second derivative may become unbounded.  
	\end{abstract}
	
	\maketitle
	
	\section{Introduction}
	
	This paper answers a question in the asymptotic theory of differential equations, which, although elementary, does not appear to have a clearcut resolution in the literature. The set--up is the following: suppose we have the forced second--order linear autonomous differential equation 
	\begin{equation} \label{eq.x}
		x''(t) + ax'(t) + bx(t) = f(t), \quad t\geq 0,
	\end{equation}
	where $a$ and $b$ are real constants, and the forcing function is to be locally integrable and measurable. Granted initial conditions $x(0)=\xi_0$ and $x'(0)=\xi_1$, there is a unique continuous solution to this equation, defined on $[0,\infty)$. Furthermore, this function is also in twice differentiable.  
	
	The following question naturally arises, especially in applications. Suppose that the underlying unperturbed equation 
	\begin{equation} \label{eq.u}
		u''(t)+au'(t)+bu(t)=0, \quad t\geq 0	
	\end{equation}
	with initial conditions $u(0)=\xi_0$ and $u'(0)=\xi_1$ is such that 
	$u(t;\xi_0,\xi_1)\to 0$ as $t \to\infty$ for any choice of initial conditions. It is well--known that this is characterised by the parametric conditions $a>0$, $b>0$, and that these conditions imply that $u$ and its derivatives decay exponentially fast. Specifically, there exist $\xi_0$-- and $\xi_1$--independent constants  $\alpha>0$ and $C>0$ such that 
	\[
	|u^{(j)}(t,\xi_0,\xi_1)|\leq C(|\xi_0|+|\xi_1|)e^{-\alpha t}, \quad t\geq 0, \quad j=0,1,2.
	\] 
	The question now naturally presents itself: what are good conditions on the forcing function $f$, so that $x$, the solution of the perturbed equation \eqref{eq.x}, also obeys $x(t)$ as $t\to\infty$? In applications, (especially in mechanics), as well as in the general theory, we may also ask what $f$ may need to satisfy so that the state $x(t)$ and its velocity $x'(t)$ both tend to zero as $t\to\infty$, since both $u(t)$ and $u'(t)$ tend to zero as $t\to\infty$. 
	
	What is very--well understood is that the condition $f(t)\to 0$ as $t\to\infty$ is \textit{sufficient} to ensure that $x(t)\to 0$ as $t\to\infty$. By writing down a variation of constants representation for $x$, however, we can see that this is not necessary, since certain irregular forcing functions do not obey $f(t)\to 0$ as $t\to\infty$, but do satisfy the condition $f\in L^1(0,\infty)$, and this latter condition is sufficient to ensure $x(t)\to 0$. Furthermore, the condition $f(t)\to 0$ as $t\to\infty$ is sufficient to ensure that $x(t)\to 0$ and $x'(t)\to 0$ as $t\to\infty$, but once again, it may not be a necessary condition. 
	
	In this paper, we characterise minimal conditions on $f$, which are independent of the structural parameters $a$ and $b$, and the initial conditions $\xi_0$ and $\xi_1$ which characterise the asymptotic convergence of $x$ to zero, and the convergence of $x$ and $x'$ to zero. We also give theorems which characterise when solutions of the differential equation tend to zero, are bounded but non--convergent, or unbounded. The results are inspired by recent work of the first author and Lawless, papers of Strauss and Yorke, and analysis for Volterra differential equations by Gripenberg, Londen and Staffans. To give a flavour of the results, we state a stability result now. Introduce the function $y_2$, given by 
	\begin{equation} \label{eq.y2}
		y_2(t)=\int_0^t (t-s)e^{-(t-s)}f(s)\,ds, \quad t\geq 0,
	\end{equation}
	and notice that it is independent of $\xi_0$, $\xi_1$ and the parameters $a$ and $b$. We also introduce the two--parameter functional of $f$ given by 
	\begin{equation} \label{eq.Ftheta1theta2}
		F_{(\theta_1,\theta_2)}(t)=\int_{(t-\theta_1)^+}^t \int_{(s-\theta_2)^+}^s f(u)\,du \,ds, \quad t\geq 0, \quad \theta_1,\theta_2 \in [0,1].
	\end{equation}
	\begin{theorem} \label{thm.xxprxprprfintint} 
		Let $f\in L^1_{loc}([0,\infty);\mathbb{R})$ and let $x$ be the unique continuous solution of the perturbed differential equation \eqref{eq.x}. Suppose moreover that $a>0$, $b>0$. Let $F_{(\theta_1,\theta_2)}$ be defined by \eqref{eq.Ftheta1theta2} and $y_2$ by \eqref{eq.y2}. Then the following are equivalent:
		\begin{itemize}
			\item[(i)] For every $(\xi_0,\xi_1)\in \mathbb{R}^2$, the solution $x(\cdot;\xi_0,\xi_1)$ of \eqref{eq.x} obeys $x(t;\xi_0,\xi_1)\to 0$ as $t\to\infty$;
			\item[(ii)] $y_2(t)\to 0$ as $t\to\infty$; 
			\item[(iii)] $F_{(\theta_1,\theta_2)}(t)\to 0$ as $t\to\infty$ uniformly in $(\theta_1,\theta_2)\in [0,1]^2$;
			\item[(iv)]  $F_{(\theta_1,\theta_2)}(t)\to 0$ as $t\to\infty$ for each  $(\theta_1,\theta_2)\in [0,1]^2$.
		\end{itemize}
	\end{theorem}
	It can readily be seen that the conditions on $F$ and $y_2$ are independent of the initial conditions and parameters, but simply depend on certain time averages of $f$. In fact, the time averages that appear in $F$ appear unweighted, and independent of the distant past. 
	
	As a taste of how $f$ may be very badly behaved, and yet $x(t)$ still tends to zero as $t\to\infty$, suppose that 
	\[
	f(t)=-4e^{3t}\sin(e^{2t}-1)+10e^t \cos(e^{2t}-1)+2e^{-t}\sin(e^{2t}-1), \quad t\geq 0,
	\] 
	and let $x$ be the solution of $x''(t)+5x'(t)+6x(t)=f(t)$ for $t\geq 0$ with $x(0)=0$, $x'(0)=2$. It is easy to verify that, despite the rapidly growing amplitude of $x(t)$, the exponentially decaying function  $x(t)=e^{-t}\sin(e^{2t}-1)$ for $t\geq 0$ solves the differential equation. On the other hand, the derivatives $x'(t)$ and $x''(t)$ are exponentially unbounded, due to the exponential unboundedness of $f$.   
	
	The idea of characterising asymptotic behaviour in differential systems by averaged conditions of the type used here goes back to 
	Strauss and Yorke \cite{SY67b,SY67a}, and Gripenberg, Londen and Staffans \cite{GLS}. Recently, Appleby and Lawless, in a sequence of papers, have shown how these results can be extended to stochastic systems and systems with memory \cite{AL:2023(AppliedMathLetters)} \cite{AL:2023(AppliedNumMath)}, \cite{AL:2024}, \cite{AL:panto24}.
	
	\section{Representation of Solutions; Characterisation of Convergence of $x$ and $x'$}
	In order to make our presentation precise and self--contained, and to understand how strengthening pointwise conditions on $f$ can give stablity results for the higher derivatives of $x$, we write down representations of the solution of \eqref{eq.x} in terms of functions which are explicitly $f$--independent, or parameter independent. In doing this, we present a result which is self--evident (Theorem~\ref{thm.xf}), as well as one that is partially available in the literature already (Theorem~\ref{thm.xxprintf}). Our rationale for stating and proving them, however, is to show how our main Theorem~\ref{thm.xxprxprprfintint} is a natural development of these results. 
	
	Let $\{e_1,e_2\}$ be the standard basis vectors in $\mathbb{R}^2$, and make the following standard definitions, which enable us to write the solution of the scalar second order equation in terms of a first order system: define
	\[
	X(t)= x(t)e_1+x'(t)e_2, \quad 
	F(t)=f(t)e_2, \quad t\geq 0
	\] 
	and $\xi=\xi_0 e_1+\xi_1 e_2$. Also define the matrix
	\[
	A=\left(\begin{matrix}
		0 & 1 \\
		-b & -a
	\end{matrix}\right),
	\]
	and define the fundamental matrix solution $\Phi(t)\in \mathbb{R}^{2\times 2}$ by 
	\begin{equation} \label{eq.Phi}
	\Phi'(t)= A \Phi(t), \quad t\geq 0; \quad \Phi(0)=I_2,
	\end{equation}
	where $I_2$ is the $2\times 2$ identity matrix. Then we may write 
	\[
	X(t)=\Phi(t)\xi + \int_0^t \Phi(t-s)F(s)\,ds, \quad t\geq 0.
	\] 
	It is useful to have notation for matrix and vector norms. Recall the $1$--norm on $\mathbb{R}^2$, where for  $v=(v_1,v_2)^T$, we have  $\|(v_1,v_2)\|_1:=|v_1|+|v_2|$, and the induced matrix $1$-- norm for $B$ in $\mathbb{R}^{2\times 2}$ is  $\|B\|_1=\max_{\|v\|_1=1}\|Bv\|_1$. This means that $\|Bv\|_1\leq \|B\|_1\|v\|_1$ for all $v\in \mathbb{R}^2$. 
	
	The conditions $a>0$ and $b>0$ ensure that there exist constants $K>0$ and $\alpha>0$ such that $\|\Phi(t)\|_1\leq Ke^{-\alpha t}$ for all $t\geq 0$. Write concretely 
	\[
	\Phi(t)=\left(\begin{matrix}
		\phi_{11}(t) & \phi_{12}(t) \\
		\phi_{21}(t) & \phi_{22}(t)
	\end{matrix}\right),
	\]    
	and write $k(t)=\phi_{12}(t)$. Since $\Phi$ obeys the differential equation \eqref{eq.Phi}, it is easy to show that 
	\begin{equation} \label{eq.k}
		k''(t)+ak'(t)+bk(t)=0, \quad t\geq 0; \quad k(0)=0, \quad k'(0)=1.
	\end{equation}
	Since $a>0$, notice that $k''(0)=-a<0$. 
	Moreover, focusing on the first component of $X$, we see that 
	\[
	x(t)=\phi_{11}(t)\xi_0+\phi_{12}(t)\xi_1+\int_0^t k(t-s)f(s)\,ds, \quad t\geq 0.
	\]
	Since each for each $i,j\in \{1,2\}$, $|\phi_{ij}(t)|\leq Ke^{-\alpha t}$ for all $t\geq 0$, and $k(t)=\phi_{12}(t)$, $k'(t)=\phi_{12}'(t)=\phi_{22}(t)$ and $k''(t)=-ak'(t)-bk(t)$, we see that there is a $C>0$ such that 
	\[
	|k^{(j)}(t)|\leq Ce^{-\alpha t}, \quad t\geq 0, \quad j=0,1,2.
	\]
	Define $x_H(t)=\phi_{11}(t)\xi_0+\phi_{12}(t) \xi_1$ for $t\geq 0$, and notice that 
	\[
	|x_H(t)|\leq K(|\xi_0|+|\xi_1|)e^{-\alpha t}, \quad t\geq 0.
	\]
	Notice that $x_H$ is the unique solution of the initial value problem 
	\begin{equation} \label{eq.xHeqn}
		x_H''(t)+ax_H'(t)+bx_H(t)=0, \quad t\geq 0; \quad x_H(0)=\xi_0, \quad x_H'(0)=\xi_1,
	\end{equation}
	and we have 
	\begin{equation} \label{eq.xrepkastf}
		x(t)=x_H(t)+\int_0^t k(t-s)f(s)\,ds, \quad t\geq 0,
	\end{equation}
	the question as to whether $x(t)$ tends to zero, is bounded but does not tend to zero, or is unbounded, hinges completely on the function $x_0(t)=(k\ast f)(t)$ for $t\geq 0$. Notice that $x_0$ is the solution of the same second order equation as $x$, but with $x_0(0)=0$, $x_0'(0)=0$. 
	
	We introduce one final object. Let $Y(t)=y_0(t)e_1+y_1(t)e_2$ where $Y(0)=0$ and $y_0'(t)=-y_0(t)$ and $y_1'(t)=-y_1(t)+f(t)$ for $t\geq 0$. Then $y_0(t)=0$ for all $t\geq 0$, 
	\begin{equation} \label{eq.y1}
		y_1(t)=\int_0^t e^{-(t-s)}f(s)\,ds, \quad t\geq 0,
	\end{equation}
	and $Y'(t)=-Y(t)+F(t)$ for $t\geq 0$. Define $Z(t)=X(t)-Y(t)$ for $t\geq 0$. Then $Z'(t)=AX(t)+Y(t)$ for $t\geq 0$, or $Z'(t)=AZ(t)+(I+A)Y(t)$ for $t\geq 0$. By variation of constants, we have 
	\begin{equation} \label{eq.Xvoc}
		X(t)=Y(t)+Z(t)=Y(t)+\Phi(t)\xi+\int_0^t \Phi(t-s)(I+A)Y(s)\,ds, \quad t\geq 0.
	\end{equation}
	Since $Y(t)=y_1(t)e_2$, \textit{This formula shows that $x$ and $x'$ can be written in terms of $f$ purely through the functional $y_1$}.
	Therefore, we see that if $Y(t)\to 0$ as $t\to\infty$, then $X(t)\to 0$ as $t\to\infty$. Since $Y(t)=y_1(t)e_2$, and $X(t)=x(t)e_1+x'(t)e_2$, we have shown that $y_1(t)\to 0$ as $t\to\infty$ implies $x(t)\to 0$ and $x'(t)\to 0$ as $t\to\infty$. Conversely, suppose that $x(t)\to 0$ and $x'(t)\to 0$ as $t\to\infty$; writing $\exp_1(t)=e^{-t}$, we see that 
	\[
	y_1=\exp_1\ast f = \exp_1
	\ast (x''+ax'+bx)
	=\exp_1*x''+a\exp_1*x'+b\exp_1*x.
	\]
	The last two terms on the righthand side are convolutions of the integrable function $\exp_1$ and functions which tend to zero as $t\to \infty$, and hence both convolutions tend to zero. 
	As for the first term, using integration by parts, we have 
	\[
	\int_0^t e^{-(t-s)}x''(s)\,ds
	= x'(t)-e^{-t}x'(0) - \int_0^t e^{-(t-s)}x'(s)\,ds,
	\]	
	and all three terms on the righthand side tend to zero as $t\to\infty$. Therefore, we have the representation
	\begin{equation} \label{eq.y1xrep}
		y_1(t)=x'(t)-e^{-t}x'(0) - \int_0^t e^{-(t-s)}x'(s)\,ds+a\int_0^t e^{-(t-s)}x'(s)\,ds + b \int_0^t e^{-(t-s)}x(s)\,ds,
	\end{equation}
	and we see that $x(t)\to 0$ and $x'(t)\to 0$ as $t\to\infty$ implies $y_1(t)\to 0$ as $t\to\infty$. Therefore we have just proven 
	\[
	x(t)\to 0, \, x'(t)\to 0, \quad t\to\infty \quad \Longleftrightarrow \quad y_1(t)\to 0, \quad t\to\infty. 
	\]
	The representations \eqref{eq.Xvoc} and \eqref{eq.y1xrep} show that $x$ and $x'$ can be written purely in terms of $y_1$, and vice versa. Since all convolutions in these formulae are taken with integrable functions ($\exp_1$, or entries of $\Phi$), this suggests that we can prove  $x$ and $x'$ lie in a nice vector space $V$ that is closed with respect to convolution, if and only if $y_1$ is in $V$.
	
	A characterisation of when $y_1(t)\to 0$ as $t\to\infty$ is given in e.g., Appleby and Lawless and tacitly in Gripenberg, Londen and Staffans, building on earlier formulations of Strauss and Yorke. It will also be of importance in the analysis of the second order differential equation. Introduce the one--parameter family of functions 
	\begin{equation} \label{eq.ftheta}
		f_\theta(t):=\int_{(t-\theta)^+}^t f(s)\,ds, \quad t\geq 0, \quad \theta\in [0,1],
	\end{equation}
	and the function 
	\begin{equation} \label{eq.deltaf}
		(\delta f)(t):=f(t)-f_1(t), \quad t\geq 0.
	\end{equation}
	From this, a Fubini argument enables us to show that 
	\begin{equation}
		\int_0^t (\delta f)(s) \, ds = \int_0^1 f_\theta(t) \, d\theta, \quad t\geq 0.
	\end{equation}
	Gripenberg, Londen and Staffans show that if $f_\theta(t)\to 0$ as $t\to\infty$ for each fixed $\theta\in [0,1]$, then there exists a $\delta\in (0,1)$, a $K_1>0$ and a $T_1>0$ such that  
	\[
	\sup_{\theta\in [0,\delta]}|f_\theta(t)|\leq K_1, \quad t\geq T_1,
	\]
	from which it can be shown that 
	\[
	\sup_{\theta\in [0,1]}|f_\theta(t)|\leq K_2, \quad t\geq T_2
	\]
	Thus, the condition $f_\theta(t)\to 0$ as $t\to\infty$ for each $\theta\in [0,1]$ together with this uniform boundedness, implies 
	\[
	F_1(t):=\int_0^t (\delta f)(s)\,ds =\int_0^1 f_\theta(t)\,d\theta \to 0, \quad t\to\infty,
	\]
	by means of the Dominated Convergence Theorem. Finally, using integration by parts, we get:
	\[
	y_1(t)=\int_0^t e^{-(t-s)}f_1(s)\,ds + \int_0^t e^{-(t-s)}(\delta f)(s)\,ds 
	=\int_0^t e^{-(t-s)}f_1(s)\,ds + F_1(t)-\int_0^t e^{-(t-s)}F_1(s)\,ds,
	\]
	and since $F_1(t)\to 0$ as $t\to\infty$, and $f_1(t)\to 0$ by hypotheses, the right hand side tends to zero as $t\to\infty$. Consequently, we have shown that 
	\[
	f_\theta(t)\to 0 \quad t\to\infty \quad \text{ for each $\theta\in [0,1]$} \quad \Longrightarrow \quad y_1(t)\to 0, \quad t\to\infty.
	\] 
	On the other hand, taking $t\geq 1$ and $\theta\in [0,1]$, integrating the differential equation for $y_1$ over $[(t-\theta)_+,t] = [t-\theta,t]$, we see that 
	\begin{equation} \label{eq.fthetarepy1}
	f_\theta(t)=y_1(t)-y_1(t-\theta)+\int_{t-\theta}^t y_1(s)\,ds.
	\end{equation}
	Therefore, if $y_1(t)\to 0$ as $t\to\infty$, it follows that $f_\theta(t)\to 0$ as $t\to\infty$ uniformly for $\theta\in [0,1]$ i.e., 
	\[
	y_1(t)\to 0, \quad t\to\infty \quad  \Longrightarrow \quad  f_\theta(t)\to 0, \quad t\to\infty \text{ uniformly for $\theta\in [0,1]$}.
	\]
	Clearly, the uniform convergence in $\theta$ of $f_\theta(t)$ implies the pointwise convergence. Therefore, we can collect together the above implications  into a result characterising the convergence of the the solution of the second order equation, and its derivative:
	\begin{theorem} \label{thm.xxprintf} 
		Let $f\in L^1_{loc}([0,\infty);\mathbb{R})$ and let $x$ be the unique continuous solution of the perturbed differential equation \eqref{eq.x}. Suppose moreover that $a>0$, $b>0$. Let $f_\theta$ be defined by \eqref{eq.ftheta} and $y_1$ by \eqref{eq.y1}. Then the following are equivalent:
		\begin{itemize}
			\item[(i)] For every $(\xi_0,\xi_1)\in \mathbb{R}^2$, the solution $x(\cdot;\xi_0,\xi_1)$ of \eqref{eq.x} obeys $x(t;\xi_0,\xi_1)\to 0$, $x'(t;\xi_0,\xi_1)\to 0$ as $t\to\infty$;
			\item[(ii)] $y_1(t)\to 0$ as $t\to\infty$; 
			\item[(iii)] $f_\theta (t)\to 0$ as $t\to\infty$ uniformly in $\theta\in [0,1]$;
			\item[(iv)]  $f_\theta(t)\to 0$ as $t\to\infty$ for each  $\theta\in [0,1]$.
		\end{itemize}
	\end{theorem}
	We note also that the representations \eqref{eq.Xvoc}, \eqref{eq.y1xrep}, and \eqref{eq.fthetarepy1} enable us to characterise the boundedness of $x$ and $x'$. The proof is similar to that of Theorem~\ref{thm.xxprintf} above, and we leave the proof to the reader, noting that the argument exploits repeatedly the fact that the convolution of an integrable function with a bounded function is bounded. 
		\begin{theorem} \label{thm.xxprintfbdd} 
		Let $f\in L^1_{loc}([0,\infty);\mathbb{R})$ and let $x$ be the unique continuous solution of the perturbed differential equation \eqref{eq.x}. Suppose moreover that $a>0$, $b>0$. Let $f_\theta$ be defined by \eqref{eq.ftheta} and $y_1$ by \eqref{eq.y1}. Then the following are equivalent:
		\begin{itemize}
			\item[(i)] For every $(\xi_0,\xi_1)\in \mathbb{R}^2$, the solution $x(\cdot;\xi_0,\xi_1)$ of \eqref{eq.x} is such that there is $B>0$ such that obeys $|x(t;\xi_0,\xi_1)|\leq B$,   $|x'(t;\xi_0,\xi_1)|\leq B$ for all $t\geq 0$. 
			\item[(ii)] $y_1$ is bounded; 
			\item[(iii)] $t\mapsto f_\theta (t)$ is uniformly bounded in $\theta\in [0,1]$ i.e., there is $B'>0$ such that $|f_\theta(t)|\leq B'$ for all $t\geq 0$ and $\theta\in [0,1]$.
		\end{itemize}
	\end{theorem}
	
	Now, for completeness, we ask: what characterises the situation where $x, x'$ and $x''$ tend to zero? The answer, obviously, is that we should have $f(t)\to 0$ as $t\to\infty$. Clearly, since $f(t)=x''(t)+ax'(t)+bx(t)$, the convergence of $x$ and its derivatives forces $f(t)\to 0$ as $t\to\infty$. On the other hand, we see that $f(t)\to 0$ as $t\to\infty$ implies $F(t)\to 0$ as $t\to\infty$, and therefore $X(t)\to 0$ as $t\to\infty$, which is equivalent to $x(t)\to 0$ as $t\to\infty$ and $x'(t)\to 0$ as $t\to\infty$. Since $f(t)\to 0$ as $t\to\infty$, and $x''(t)=f(t)-ax'(t)-bx(t)$, we see that $x''(t)\to 0$ as $t\to\infty$. Therefore, we have a characterisation of when $x$ and its derivatives tend to zero:
	\begin{theorem} \label{thm.xf} 
		Let $f\in C([0,\infty);\mathbb{R})$ and let $x$ be the unique continuous solution of the perturbed differential equation \eqref{eq.x}. Suppose moreover that $a>0$, $b>0$. Then the following are equivalent:
		\begin{itemize}
			\item[(i)] For every $(\xi_0,\xi_1)\in \mathbb{R}^2$, the solution $x(\cdot;\xi_0,\xi_1)$ of \eqref{eq.x} obeys $x(t;\xi_0,\xi_1)\to 0$, $x'(t;\xi_0,\xi_1)\to 0$, $x''(t;\xi_0,\xi_1)\to 0$ as $t\to\infty$;
			\item[(ii)] $f(t)\to 0$ as $t\to\infty$; 
		\end{itemize}
	\end{theorem} 
	Parallel to Theorem~\ref{thm.xxprintfbdd} above, a characterisation of when $x$, $x'$ and $x''$ are bounded follows along almost identical lines, and the characterising condition in this case is that $f$ is bounded.
	
	Taking Theorem~\ref{thm.xf}, Theorem~\ref{thm.xxprintf} and Theorem~\ref{thm.xxprxprprfintint} together, an appealing picture emerges. The ``standard'' pointwise condition $f(t)\to 0$ as $t\to\infty$ which gives  convergenece of $x$ to zero, is sufficiently strong to ensure that the derivative and second derivative also tend to zero, and this condition is exactly what is necessary to  characterise this strengthened convergence of the system. If we take the weaker first order integral condition hypotheses that $f_\theta(t)$ tends to zero, or $y_1(t)\to 0$ as $t\to\infty$, we lose control over the second derivative of $x$, but these are precisely the conditions that capture the convergence of $x$ and $x'$ to zero. 
	This explains the nature of Theorem~\ref{thm.xxprxprprfintint}: notice that $y_2(t)=(\exp_1\ast y_1)(t)$, so that $y_2$ implicitly involves two integrations, as does $F_{\theta_1,\theta_2}$; this allows us to retain control over $x$, but at the possible cost of losing control over the two derivatives $x'$ and $x''$.  
	
	In other words, each time we weaken the hypotheses on $f$ by integrating an extra time, we lose control over an extra derivative of $x$. But if $y_2$ or $F$ do not tend to zero, the asymptotic convergence of the underlying second order differential equation for $u$ is lost. Focusing on the $y$'s, we see in other words, with the standard recursive convolution notation, $g^{(*0)}:=\delta_0$ ($\delta_0$ is the Dirac delta distribution at zero), $g^{(*1)}:=g$ and 
	\[
	g^{(*n)}=g^{(*(n-1))}*g, \quad n\geq 2,
	\]
	that  the first parts of Theorems  \ref{thm.xf}, \ref{thm.xxprintf} and \ref{thm.xxprxprprfintint} can be unified into the single equivalence 
	\[
	y_j(t):=(\exp_1^{(*j)}*f)(t) \to 0 \quad t\to\infty, \quad \Longleftrightarrow \quad x^{(l)}(t)\to 0, \quad t\to\infty, \quad l=0,\ldots,2-j.
	\]
	where we used the convolution notational convention to define $y_0(t):=f(t)=(\delta_0\ast f)(t)=(\exp_1^{(*0)}*f)(t)$.

	In a later section, we will use representations of $x$, $x'$, $x''$ in terms of $y_1$, $y_2$ etc to show how we may characterise the boundedness and unboundedness of solutions as well, 
	
	\section{Representations of $x$, $y_2$ and $F$}
	This section is devoted to showing how $x_0$, $y_2$ and $F$ can be represented in terms of each other.
	
	We start by showing how $x_0=k\ast f$ and $y_2$ can be written in terms of one another. Start by observing that $y_2$ is the unique continuous solution of the second--order initial value problem 
	\[
	y''_2(t) + 2y_2'(t) + y_2(t) = f(t), \quad t \geq 0; \quad y_2(0)=y'_2(0)=0.
	\]
	The following result emerged from our calculations in the last section. 
	\begin{theorem}
		Let $a, b \in \mathbb{R}$. Let $k$ obey \eqref{eq.k}. Then $x_0=k\ast f$ is the unique solution to $x_0''(t)+ax_0'(t)+bx_0(t)=f(t)$ with $x_0(0)=0$, $x_0'(0)=0$. 
	\end{theorem}
	
	\begin{theorem} \label{thm.x0y2y2x0} 
		Let $k$ obey \eqref{eq.k}, $x_0=k\ast f$ and $y_2$ be given by \eqref{eq.y2}. Then
		\begin{equation} \label{eq.x0y2}
			x_0(t) = y_2(t) + \left( (k'' + 2k' + k) * y_2 \right)(t), \quad t\geq 0,
		\end{equation}
		and, with $g(t)=te^{-t}$, 
		\begin{equation} \label{eq.y2repx0}
			y_2(t) = x_0(t) + \left((g''+ag'+bg) * x_0 \right)(t)=:x_0(t)+(\phi\ast x_0)t, \quad t\geq 0,
		\end{equation}
		where $\phi$ is explicitly given by
		\begin{equation} \label{eq.phi}
			\phi(t) = (a-2)e^{-t} + (b-a+1)t e^{-t}, \quad t\geq 0.
		\end{equation}
	\end{theorem}
	\begin{proof}
		We prove \eqref{eq.x0y2}; the proof of \eqref{eq.y2repx0} is essentially identical, and left to the reader, exploiting the fact that $y_2(t)=(g\ast f)(t)$, where $g(t)=te^{-t}$ satisfies the differential equation $g''+2g'+g=0$ with $g(0)=0$ and $g'(0)=1$. Because of this, we can reverse the roles of $x_0$ and $y_2$, where, in the proof below, the role of $k$ is fulfilled by $g$.  
		
		Since $x_0 = k * f$ and substitute $f = y''_2 + 2y'_2 + y_2$, we have 
		\begin{equation} \label{eq.x0repyypryprpr}
			x_0 = k * (y''_2 + 2y'_2 + y_2) = (k * y''_2) + 2(k * y'_2) + (k * y_2).
		\end{equation}
		We consider the second term on the right hand side of \eqref{eq.x0repyypryprpr}. Integrating by parts gives
		\begin{equation} \label{eq.ky2pr}
			\int_0^t k(t-s)y_2'(s)\,ds = k(0) y_2(t)-k(t)y_2(0) + \int_{0}^t k'(t-s)y_2(s)\,ds = (k'\ast y_2)(t),
		\end{equation}
		using $y_2(0)=k(0)=0$. For the first term on the righthand side of \eqref{eq.x0repyypryprpr}, we need to integrate by parts twice. 
		This gives 
		\[
		\int_0^t k(t-s)y_2''(s)\,ds
		= k(0)y_2'(t)-k(t)y_2'(0)+\int_0^t k'(t-s)y_2'(s)\,ds
		=(k'\ast y_2')(t),
		\]	
		since $k(0)=0$ and $y_2'(0)=0$, and integrating by parts again, we get 
		\begin{equation} \label{eq.kpry2pr}
			(k\ast y_2'')(t)=\int_0^t k'(t-s)y_2'(s)\,ds
			=k'(0)y_2(t)-k'(t)y_2(0)+\int_0^t k''(t-s)y_2(s)\,ds = y_2(t)+(k''\ast y_2)(t),
		\end{equation}
		since $k'(0)=1$ and $y_2(0)=0$. Combining this and \eqref{eq.ky2pr} with \eqref{eq.x0repyypryprpr}, we get the claimed formula \eqref{eq.x0y2}.
	\end{proof}
	
	This result means we can immediately prove the first equivalence (i) $\Longleftrightarrow$ (ii) in Theorem~\ref{thm.xxprxprprfintint}. 
	Since $a$ and $b$ are positive and $k$ obeys \eqref{eq.k}, we have that $k''$, $k'$ and $k$ are in $L^1(0,\infty)$, and moreover that $x_H$ given by \eqref{eq.xHeqn} tends to zero as $t\to\infty$. We have $x(t;\xi_0,\xi_1)=x_H(t)+x_0(t)$. If $y_2(t)\to 0$ as $t\to\infty$, it therefore follows from the integrability of $k^{(j)}$ for $j=0,1,2$, and the representation \eqref{eq.x0y2} that $x_0(t)\to 0$ as $t\to\infty$, and therefore that $x(t;\xi_0,\xi_1)\to 0$ as $t\to\infty$. Conversely, if $x(t;\xi_0,\xi_1)\to 0$ as $t\to\infty$, since $x_H(t)\to 0$ as $t\to\infty$, it follows that $x_0(t)\to 0$ as $t\to\infty$. Therefore, by \eqref{eq.y2repx0} and the integrability of the function $\phi$ in \eqref{eq.phi}, we see that $y_2(t)\to 0$ as $t\to\infty$.
	
	Next, we give a representation of $F$ in terms of $y_2$. 	
	
	\begin{theorem} \label{thm.Frepy2}
		Let $y_2$ be given by \eqref{eq.y2}, and $F$ by \eqref{eq.Ftheta1theta2}. Let $\theta_1,\theta_2\in [0,1]$ and $t\geq 2$. Then 
		\begin{multline} \label{eq.Frepy2}
			F_{\theta_1,\theta_2}(t)
			=
			y_2(t) - y_2(t-\theta_1) - y_2(t-\theta_2) + y_2(t-\theta_1-\theta_2) \\+ 2 \int_{t-\theta_1}^t (y_2(s) - y_2(s-\theta_2)) \, ds + \int_{t-\theta_1}^t \int_{s-\theta_2}^s y_2(u) \, du \, ds.
		\end{multline}
	\end{theorem}
	\begin{proof}
		Since $y_2$ satisfies the second order differential equation  $y_2''(u) + 2y_2'(u) + y_2(u) = f(u)$, integrate with respect to $u$ over the interval $[s-\theta_2, s]$, assuming $s \ge 1$:
		\[
		\int_{s-\theta_2}^s y''_2(u) \, du + 2\int_{s-\theta_2}^s y'_2(u) \, du + \int_{s-\theta_2}^s y_2(u) \, du = \int_{s-\theta_2}^s f(u) \, du,
		\]
		which gives 
		\[
		\int_{s-\theta_2}^s f(u) \, du = 
		[y'_2(s) - y'_2(s-\theta_2)] + 2[y_2(s) - y_2(s-\theta_2)] + \int_{s-\theta_2}^s y_2(u) \, du. 
		\]
		Next, we integrate this result with respect to $s$ over the interval $[t-\theta_1, t]$, assuming $t \ge 2$:
		\[
		\int_{t-\theta_1}^t \int_{s-\theta_2}^s f(u) \, du \, ds=
		\int_{t-\theta_1}^t [y'_2(s) - y'_2(s-\theta_2)] \, ds + 2\int_{t-\theta_1}^t [y_2(s) - y_2(s-\theta_2)] \, ds + \int_{t-\theta_1}^t \int_{s-\theta_2}^s y_2(u) \, du \, ds.
		\]
		Since the lower limits in the integral on the lefthand side are non-negative (because $t\geq 2$ and $\theta_1,\theta_2\in [0,1]$), the left hand side is exactly $F_{\theta_1,\theta_2}(t)$.  
		On the righthand side, we  evaluate the first term explicitly:
		\begin{align*}
			\int_{t-\theta_1}^t (y'_2(s) - y'_2(s-\theta_2)) \, ds 
			&= y_2(t) - y_2(t-\theta_1) - y_2(t-\theta_2) + y_2(t-\theta_1-\theta_2). 
		\end{align*}
		Substituting this back into the integrated equation yields:
		\begin{multline*}
			F_{\theta_1,\theta_2}(t)=
			y_2(t) - y_2(t-\theta_1) - y_2(t-\theta_2) + y_2(t-\theta_1-\theta_2) \\
			+ 2 \int_{t-\theta_1}^t (y_2(s) - y_2(s-\theta_2)) \, ds + \int_{t-\theta_1}^t \int_{s-\theta_2}^s y_2(u) \, du \, ds,
		\end{multline*}
		as required. 
	\end{proof}
	An immediate consequence of this representation is that $y_2(t)\to 0$ as $t\to\infty$ implies that $F_{\theta_1,\theta_2}(t)$ converges uniformly (in $(\theta_1,\theta_2)\in [0,1]^2$) to $0$ as $t\to\infty$, and therefore in turn that $F_{\theta_1,\theta_2}(t)$ converges pointwise. This means we have proven the implications (ii) $\Rightarrow$ (iii) $\Rightarrow$ (iv) in Theorem~\ref{thm.xxprxprprfintint}. 
	It is self--evident that since $x$ satisfies a second order differential equation, we can equally derive the following representation of $F$ in terms of $x$ for $t\geq 2$ and $\theta_1$ and $\theta_2$ in $[0,1]$:
	\begin{multline} \label{eq.Frepx}
		F_{\theta_1,\theta_2}(t)=
		x(t) - x(t-\theta_1) - x(t-\theta_2) + x(t-\theta_1-\theta_2) \\
		+ a \int_{t-\theta_1}^t (x(s) - x(s-\theta_2)) \, ds + b\int_{t-\theta_1}^t \int_{s-\theta_2}^s x(u) \, du \, ds.
	\end{multline}
	
	\section{Representation of $x_0$ and $y$ in terms of $F$}
	We state as a lemma a result we alluded to at the beginning of the paper. It was originally proven in Gripenberg, Londen and Staffans, with more details given in Appleby and Lawless. 
	\begin{lemma} \label{lemma:delta_integral}
		Suppose $g \in L_{loc}^1[0, \infty)$. For each $\theta \in [0, 1]$, define the moving average integral:
		\[
		g_\theta(t) = \int_{(t-\theta)^+}^t g(s) \, ds, \quad t \ge 0.
		\]
		Define the difference function $(\delta g)(t) = g(t) - g_1(t)$.
		Then for all $t \ge 0$:
		\begin{equation}
			\int_0^t (\delta g)(s) \, ds = \int_0^1 g_\theta(t) \, d\theta.
		\end{equation}
	\end{lemma}
	We now use this formula to obtain a representation for a convolution integral.
	Let $f \in L^1_{\text{loc}}([0, \infty))$ be a locally integrable function. 
	We now derive explicit formulas connecting the convolution $k * f$ to the functional $F_{(\theta_1,\theta_2)}$, eliminating any direct dependence on $f$. 
	
	\begin{theorem} \label{thm.xrepF}
		Let $k \in C^2([0, \infty))$ with $k(0)=0$ and $k'(0)=1$. Let $f$ be locally integrable, and $F$ given by \eqref{eq.Ftheta1theta2}. 
		Then $(k * f)(t)$ is given by:
		\begin{multline} \label{eq.x0y2intermsofF}
			(k * f)(t) = (k * F_{(1,1)})(t) + \left( k' * \int_0^1 F_{(\phi, 1)} \, d\phi \right)(t) \\
			+ \left( k' * \int_0^1 F_{(1, \theta)} \, d\theta \right)(t) + \iint_{[0,1]^2} F_{(\phi, \theta)}(t) \, d\phi \, d\theta + \left( k'' * \iint_{[0,1]^2} F_{(\phi, \theta)} \, d\phi \, d\theta \right)(t).
		\end{multline}
	\end{theorem}
	We notice that this develops a consequence of Lemma~\ref{lemma:delta_integral} that we earlier hinted at, and which has been exploited by Gripenberg, Londen and Staffans and Appleby and Lawless in analysis of Volterra differential equations: suppose $k\in C^1([0,\infty))$ and $k(0)=1$. Let $f$ be locally integrable, and $f_\theta$ be given by \eqref{eq.ftheta}. Then 
	\[
	(k\ast f)(t)=(k\ast f_1)(t)+\int_0^1 f_\theta(t)\,d\theta + \left(k'\ast \int_0^1 f_\theta \,d\theta\right)(t), \quad t\geq 0.
	\] 
	Note also that we have made no assumption here that $k$ satisfies a particular differential equation, merely that is obeys certain smoothness and boundary conditions. This also renders it suitable for the asymptotic analysis of forced second--order convolution integrodifferential equations, where the solution depends on the convolution of a (possible irregular) forcing function and a differential resolvent that solves an unforced second order convolution integrodifferential equation.

	\begin{proof}
		We start with $I = k * f$ and decompose $f = f_1 + \delta f$: write
		\[ I = I_1 + I_2, \quad \text{where } I_1 = k * f_1, \quad I_2 = k * \delta f. \]
		\textbf{Step 1: Analysis of $I_2$}
		We apply integration by parts to $I_2 = \int_0^t k(t-s) (\delta f)(s) \, ds$. Using Lemma \ref{lemma:delta_integral}, we have that $b(s):=\int_0^1 f_\theta(s)\,d\theta$ implies $b'(s)=(\delta f)(s)$. Therefore 
		\[
		\int_0^t k(t-s)(\delta f)(s)\,ds
		= k(0)b(t)
		-k(t)b(0)+\int_0^t k'(t-s)b(s)\,ds
		=(k'\ast b)(t).
		\]
		Now, using the notation of Lemma~\ref{lemma:delta_integral}, we write $b = b_1 + \delta b$, so that 
		\begin{equation} \label{eq.I2}
			I_2(t)=\int_0^t k'(t-s)b_1(s)\,ds+\int_0^t k'(t-s)(\delta b)(s)\,ds.
		\end{equation}
		We now analyse these two integrals. First, we find $b_1$. Applying the definitions of $b_1$, $f_\theta$, and $F_{(1,\theta)}$ in order, we get  
		\begin{equation} \label{eq.b1}
			b_1(t) = \int_{(t-1)^+}^t b(s) \, ds = \int_0^1 \int_{(t-1)^+}^t f_\theta(s) \, ds \, d\theta
			=\int_0^1 \int_{(t-1)^+}^t
			\int_{(s-\theta)^+}^s f(u) \, du
			\,ds\,d\theta
			= \int_0^1 F_{(1,\theta)}(t)\,d\theta.
		\end{equation}
		Next, we deal with $(k'\ast (\delta b))(t)$. Using the definition of $b_\phi(t):=\int_{(t-\phi)^+}^t b(s)\,ds$ and applying Lemma \ref{lemma:delta_integral}, we have $\int_0^s (\delta b)(u)\,du = \int_0^1 b_\phi(s)\,d\phi$. Thus, if $v(s)=\int_0^1 b_\phi(s)\,d\phi$, then $v'(s)=(\delta b)(s)$. Now, integrating by parts, we get  
		\begin{equation} \label{eq.kprdelb}
			\int_0^t k'(t-s)(\delta b)(s)\,ds
			=k'(0)v(t)-k'(t)v(0)+\int_0^t k''(t-s)v(s)\,ds.
		\end{equation}
		Now, to compute $v$, we need $b_\phi$. Using in turn the definition of $b_\phi$, then the definition of $b$, Fubini's theorem, and finally the definitions of $f_\theta$ and $F$, we arrive at 
		\[ b_\phi(s) = \int_{(s-\phi)^+}^s b(u) \, du = \int_{(s-\phi)^+}^s \int_0^1 f_\theta(s)\,d\theta\,du = \int_0^1 \int_{(s-\phi)^+}^s f_\theta(u) \, du \, d\theta = \int_0^1 F_{(\phi, \theta)}(s) \, d\theta. \]
		Therefore
		\begin{equation} \label{eq.v}
			v(t)=\int_0^1 b_\phi(t)\,d\phi = \int_0^1 \int_0^1 F_{(\phi,\theta)}(t)\,d\theta\, d\phi.
		\end{equation}
		Since $F_{(\phi,\theta)}(0)=0$, we have $v(0)=0$. 
		Therefore from \eqref{eq.kprdelb}, and using the fact that $k'(0)=1$, we have  
		\begin{equation*} 
			\int_0^t k'(t-s)(\delta b)(s)\,ds
			=v(t)+\int_0^t k''(t-s)v(s)\,ds.
		\end{equation*}
		Combining this with \eqref{eq.v}, 
		\eqref{eq.b1} and \eqref{eq.I2}, we get that 
		\begin{equation} \label{eq.I2final}
			I_2(t)=\left(k'\ast \int_0^1 F_{(1,\theta)}(\cdot)\,d\theta\right)(t)+ \int_0^1 \int_0^1 F_{(\phi,\theta)}(t)\,d\theta\, d\phi
			+\left(k''\ast \int_0^1 \int_0^1 F_{(\phi,\theta)}(\cdot)\,d\theta\, d\phi\right)(t).
		\end{equation}
		
		\textbf{Step 2: Analysis of $I_1$}
		It remains to consider $I_1(t) = (k * f_1)(t)$. Let $a = f_1$. Using the notation of Lemma~\ref{lemma:delta_integral}, we write $a = a_1 + \delta a$, so that 
		\[
		I_1(t)=\int_0^t k(t-s)a_1(s)\,ds + \int_0^t k(t-s)(\delta a)(s)\,ds. 
		\]
		For the first term on the right hand side, we need an expression for $a_1$. From the definitions of $a_1$, $a$, $f_1$ and $F$, we find that  
		\[
		a_1(t) = \int_{(t-1)^+}^t a(s) \, ds = \int_{(t-1)^+}^t \int_{(s-1)^+}^s f(u) \, du \, ds = F_{(1,1)}(t),
		\]
		so $(k\ast a_1)(t) = (k\ast F_{(1,1)})(t)$. 
		
		For the second term, let $v_a(s) = \int_0^1 a_\phi(s)\,d\phi$. Then, by Lemma~\ref{lemma:delta_integral}, $v_a'(s)=(\delta a)(s)$. Integrating by parts we get 
		\begin{equation} \label{eq.kdela}
			\int_0^t k(t-s)(\delta a)(s)\,ds = k(0)v_a(t)-k(t)v_a(0)+\int_0^t k'(t-s)v_a(s)\,ds.
		\end{equation}
		Next, we find an expression for $a_\phi(s)$: by the definition of $a_\phi$, the fact $a=f_1$, and the definition of $F$, we get 
		\[ 
		a_\phi(s) = \int_{(s-\phi)^+}^s a(u) \, du = \int_{(s-\phi)^+}^s \int_{(u-1)^+}^u f(y) \, dy \, du = F_{(\phi, 1)}(s). \]
		Therefore 
		\[
		v_a(s)=\int_0^1 F_{(\phi, 1)}(s)\,d\phi,
		\]
		so as $F_{(\phi,1)}(0)=0$, $v_a(0)=0$. Using this and $k(0)=0$ in \eqref{eq.kdela} we get 
		\[
		\int_0^t k(t-s)(\delta a)(s)\,ds = (k'\ast v_a)(t) = \left(k'\ast \int_0^1 F_{(\phi, 1)}(\cdot)\,d\phi\right)(t), \quad t\geq 0.	
		\]
		Combining this formula with the fact that $k\ast a_1 = k\ast F_{(1,1)}$, we have that 
		\[
		I_1(t)= (k\ast F_{(1,1)})(t)+\left(k'\ast \int_0^1 F_{(\phi, 1)}(\cdot)\,d\phi\right)(t), \quad t\geq 0.
		\]
		Finally, combining this formula with \eqref{eq.I2final}, we see that $I=k\ast f=I_1+I_2$ has the claimed representation. 
	\end{proof} 
	We notice that $x_0=k\ast f$ fulfills all the hypotheses of the theorem, since $k$ obeys \eqref{eq.k}. Likewise, with $g(t)=te^{-t}$, we see that $g$ is in $C^2([0,\infty)$ and $g(0)=0$, $g'(0)=1$. Moreover, $y_2=g\ast f$, so $g$ can play the role of $k$ in Theorem~\ref{thm.xrepF}. Moreover, we see by the hypothesis $a>0$, $b>0$ that $k,k', k''$ are in $L^1(0,\infty)$, as are $g, g'$ and $g''$. Now, suppose that 
	$F_{(\theta_1,\theta_2)}(t)\to 0$ as $t\to\infty$ uniformly in $(\theta_1,\theta_2)\in [0,1]^2$. Therefore the functions 
	\[
	\int_0^1 F_{(\phi,1)}(t)\,d\phi, \quad \int_0^1 F_{(1,\theta)}(t)\,d\theta, \quad \text{and} \quad  \int_0^1	\int_0^1 F_{(\phi,\theta)}(t)\,d\phi \,d\theta. 
	\] 
	all tend to zero as $t\to\infty$. Therefore, since $k$, $k'$ and $k''$ are all in $L^1(0,\infty)$, all the terms on the righthand side of \eqref{eq.x0y2intermsofF} tend to zero as $t\to\infty$, so $x_0(t)\to 0$ as $t\to\infty$. Therefore $x(t;\xi_0,\xi_1)\to 0$ as $t\to\infty$; a similar argument, using the integrability of $g$, $g'$ and $g''$ gives $y_2(t)\to 0$ as $t\to\infty$. Thus, we have shown that statement (iii) in Theorem~\ref{thm.xxprxprprfintint} implies statement (i) and statement (ii), and we have already shown that statements (i) and (ii) are equivalent. Hence statements (i)--(iii) in Theorem~\ref{thm.xxprxprprfintint} are equivalent, and statement (iii) clearly implies statement (iv). The proof that statement (iv) implies statement (iii) (pointwise convergence implies uniform convergence) is given in the last section. 
	
	\section{Characterisation of Convergence, Boundedness and Unboundedness} 
	
	The representations for $x_0$ and $y_2$ in terms of $F$ furnished by Theorems~\ref{thm.xrepF}, for $F$ in terms of $y_2$ from Theorem \ref{thm.Frepy2}, and for $x_0$ in terms of $y_2$ and $y_2$ in terms of $x_0$ in Theorem~\ref{thm.x0y2y2x0} now allow us to give a complete classification of whether $x$ tends to zero, is bounded but does not tend to zero, or is unbounded. 
	To this end, we define
	\begin{align} \label{eq.Xbar}
		\bar{X}&=\limsup_{t\to\infty} |x(t)|\in [0,\infty],\\
		 \label{eq.Fbar}
		\bar{F}&=\limsup_{t\to\infty} \sup_{\theta_1,\theta_2\in [0,1]^2} |F_{(\theta_1,\theta_2)}(t)|\in [0,\infty],\\
		\label{eq.Y2bar}
		\bar{Y_2}&=\limsup_{t\to\infty} |y_2(t)|\in [0,\infty].
	\end{align}
	All these are well-define elements of the extended real line. 
	\begin{theorem}
		Suppose that $a>0$, $b>0$ and that $f\in L^1_{loc}(\mathbb{R}^+)$. Let $x$ be the unique continuous solution of \eqref{eq.x}, and suppose that $\bar{X}$ and $\bar{Y_2}$ are given by \eqref{eq.Xbar} and \eqref{eq.Y2bar}, and $y_2$ is given by \eqref{eq.y2}.
		\begin{itemize}
			\item[(i)] If $\bar{Y_2}=0$, then $\bar{X}=0$.
			\item[(ii)] If $\bar{Y_2}\in (0,\infty)$, then $\bar{X}\in (0,\infty)$.
			\item[(iii)] If $\bar{Y_2}=+\infty$, then  $\bar{X}=+\infty$.
		\end{itemize}
	\end{theorem}
	\begin{proof}
		Part (i). We have already shown that  $y_2(t)\to 0$ as $t\to\infty$ implies $x(t)\to 0$ as $t\to\infty$, which is precisely the content of part (i). 
		
		Part (ii). If $\bar{Y_2}\in (0,\infty)$, $y_2$ is bounded. By Theorem \ref{thm.x0y2y2x0}, since $k, k', k''$ are in $L^1$ it follows from  \eqref{eq.x0y2} that $x_0$ is  bounded, since it is the sum of a bounded function, and an $L^1$ function convolved with a bounded function. Since $x_H$ tends to zero, $x=x_H+x_0$ is boundeded, and so $\bar{X}<+\infty$. Suppose now by way of contradiction that $\bar{X}=0$. Then by the implication (i) $\Rightarrow$ (ii) in Theorem~\ref{thm.xxprxprprfintint}, we have that $y_2(t)\to 0$ as $t\to\infty$. But then $\bar{Y_2}=0$, contradicting the hypothesis. Hence $\bar{X}\in (0,\infty)$.
		
		Part (iii). Suppose by way of contradiction that $\bar{X}<+\infty$, so $x$ is bounded. Since $x_H$ tends to zero, this means that $x_0$ is bounded. Then, by Theorem~\ref{thm.x0y2y2x0}, specifically \eqref{eq.y2repx0}, it follows that $y_2$ is the sum of a bounded function and the convolution of an integrable function $\phi$ and the bounded function $x_0$. Therefore $y_2$ is bounded, so $\bar{Y_2}<+\infty$. But this contradicts the hypothesis, so we must have $\bar{X}=+\infty$, as required.
	\end{proof}
	
	A completely analogous classification results from using $\bar{F}$. 
	
	\begin{theorem}
		Suppose that $a>0$, $b>0$ and that $f\in L^1_{loc}(\mathbb{R}^+)$. Let $x$ be the unique continuous solution of \eqref{eq.x}, and suppose that $\bar{X}$ and $\bar{F}$ are given by \eqref{eq.Xbar} and \eqref{eq.Fbar}, and $F$ is given by \eqref{eq.Ftheta1theta2}.
		\begin{itemize}
			\item[(i)] If $\bar{F}=0$, then $\bar{X}=0$.
			\item[(ii)] If $\bar{F}\in (0,\infty)$, then $\bar{X}\in (0,\infty)$.
			\item[(iii)] If $\bar{F}=+\infty$, then  $\bar{X}=+\infty$.
		\end{itemize}
	\end{theorem}
	\begin{proof}
		Part (i). If $\bar{F}=0$, we have that $F_{(\theta_1,\theta_2)}(t)\to 0$ as $t\to\infty$ uniformly in $(\theta_1,\theta_2)$. The implication (iii) $\Rightarrow$ (i) in Theorem~\ref{thm.xxprxprprfintint} yields  $x(t)\to 0$ as $t\to\infty$, or $\bar{X}=0$, as needed.  
		
		Part (ii). If $\bar{F}\in (0,\infty)$, it means that there is $T_1>2$ such that 
		\[
		\sup_{\theta_1,\theta_2\in [0,1]^2} F_{(\theta_1,\theta_2)}(t) \leq \bar{F}+1, \quad t\geq T_1.
		\]
		As a result, each of 
		\[
		F_{(1,1)}(t), \quad 
		\int_0^1 F_{(\phi,1)}(t)\,d\phi, \quad \int_0^1 F_{(1,\theta)}(t)\,d\theta, \quad \text{and} \quad  \int_0^1	\int_0^1 F_{(\phi,\theta)}(t)\,d\phi \,d\theta. 
		\] 
		is absolutely bounded by $\bar{F}+1$ for all $t\geq T_1$. Since all of these are continuous functions, they are uniformly absolutely bounded by $F^\ast\in (0,\infty)$ on $[0,\infty)$. Therefore by Theorem~\ref{thm.xrepF}, we get the estimate
		\[
		|x_0(t)|\leq \int_0^\infty |k(s)|\,ds F^\ast+2\int_0^\infty |k'(s)|\,ds F^\ast+F^\ast+\int_0^\infty |k''(s)|\,ds F^\ast,
		\]
		by taking the triangle inequality and suprema across \eqref{eq.x0y2intermsofF}.
		Since $x_H$ tends to zero, we have that $x$ is uniformly bounded. Thus $\bar{X}<+\infty$. Suppose now by way of contradiction that $\bar{X}=0$. Then $x(t)\to 0$ as $t\to\infty$. This implies that $y_2(t)\to 0$ as $t\to\infty$ by Theorem~\ref{thm.xxprxprprfintint}. But then, by Theorem~\ref{thm.Frepy2}, we see that $\bar{Y_2}=0$ implies $\bar{F}=0$, which contradicts the hypothesis that $\bar{F}>0$. Hence $\bar{X}\in (0,\infty)$ as claimed.  
		
		Part (iii). Suppose by way of contradiction that $\bar{X}<+\infty$, so $x$ is uniformly bounded by $B$. Then by taking the triangle inequality across \eqref{eq.Frepx} for $t\geq 2$ and $\theta_1,\theta_2\in [0,1]$, we get 
		\[
		|F_{\theta_1,\theta_2}(t)|\leq 
		4B  
		+ |a| \int_{t-\theta_1}^t 2B \, ds + |b|\int_{t-\theta_1}^t \int_{s-\theta_2}^s B \, du \, ds
		\leq 4B+2B|a|+|b|B\theta_1\theta_2\leq 4B+2B|a|+B|b|.
		\]
		From this $\bar{F}<+\infty$. But this contradicts the hypothesis that $\bar{F}=+\infty$, thereby contradicting the supposition $\bar{X}<+\infty$. Hence $\bar{X}=+\infty$, as claimed. 
	\end{proof}

	\section{Uniform Convergence of $F_\Theta$}
	In this section, we strengthen the pointwise properties of $F_\Theta$ to uniform properties over the parameter space $S = [0,1]^2$. This allows us to assert the equivalence of conditions (iii) and (iv) in Theorem~\ref{thm.xxprxprprfintint}, and thereby completes the proof of the main result of the paper. For applications, this means it becomes unnecessary to prove the uniform convergence of $F_{\theta_1,\theta_2}(t)$, which may be difficult or tedious; instead it will be enough to check merely that $F_{(\theta_1,\theta_2)}(t)\to 0$ for each choice of $(\theta_1,\theta_2)\in [0,1]^2$. 
	\begin{theorem} \label{thm:uniformity}
		Let $S = [0,1]^2$. If $\lim_{t \to \infty} F_\Theta(t) = 0$ for every $\Theta \in S$, then the convergence is uniform in $\Theta$. That is,
		\[ \lim_{t \to \infty} \sup_{\Theta \in S} |F_\Theta(t)| = 0. \]
\end{theorem}
To prove this, we begin by establishing a fundamental decomposition lemma. Note that because of the positive part truncations in the definition of $F_\Theta$, the algebraic decomposition holds only when $t$ is sufficiently large to avoid the boundary at $t=0$.
\begin{lemma}[Decomposition Identity]
	Let $\Theta = (\theta_1,\theta_2) \in [0,1]^2$. Let $e_k$ be the standard basis vector for the $k$-th component. For any $\delta \ge 0$, define $\Theta' = \Theta + \delta e_k$.
	If $t \ge 2 + 2\delta$, then:
	\begin{equation} \label{eq:decomp}
		F_{\Theta'}(t) = F_{\Theta}(t) + F_{\Theta_{\delta}}(t - \theta_k),
	\end{equation}
	where $\Theta_{\delta}$ replaces the $k$--th component of $\Theta$ by $\delta$. 
\end{lemma}

\begin{proof}[Proof of Lemma]
	We choose $t$ sufficiently large so that 
	positive parts in the lower limits of integration can be dropped. Thus for $\Theta = (\theta_1, \theta_2)$, and these choices of $t$, the functional is:
	\[
	F_{(\theta_1, \theta_2)}(t) = \int_{t-\theta_1}^t \left( \int_{t_1-\theta_2}^{t_1} f(t_2) \, dt_2 \right) dt_1.
	\]
	
	\textbf{Case 1: Incrementing the first parameter ($k=1$)}
	Let $\Theta' = (\theta_1 + \delta, \theta_2)$. We replace $\theta_1$ with $\theta_1+\delta$ in the outer integral limits, and split the integral
	\begin{align*}
	F_{(\theta_1+\delta, \theta_2)}(t) 
	&= \int_{t-(\theta_1+\delta)}^t \left( \int_{t_1-\theta_2}^{t_1} f(t_2) \, dt_2 \right) dt_1\\
	&= \int_{t-\theta_1}^t \left( \int_{t_1-\theta_2}^{t_1} f(t_2) \, dt_2 \right) dt_1 + \int_{t-\theta_1-\delta}^{t-\theta_1} \left( \int_{t_1-\theta_2}^{t_1} f(t_2) \, dt_2 \right) dt_1\\
	&=:F_{(\theta_1,\theta_2)}(t)+J_1(t).
	\end{align*}
	We claim $J_1(t)=F_{(\delta, \theta_2)}(t-\theta_1)$.
	To prove this, let us evaluate $F_{(\delta, \theta_2)}$ at the shifted time $\tau = t-\theta_1$:
		\[
		F_{(\delta, \theta_2)}(\tau) = \int_{\tau-\delta}^\tau \left( \int_{t_1-\theta_2}^{t_1} f(t_2) \, dt_2 \right) dt_1.
		\]
		Substituting $\tau = t-\theta_1$, the limits become $[(t-\theta_1)-\delta, t-\theta_1]$, which $J_1(t)$.
	Thus, when we increment the first parameter, $F_{(\theta_1+\delta, \theta_2)}(t) = F_{(\theta_1, \theta_2)}(t) + F_{(\delta, \theta_2)}(t-\theta_1)$, as claimed.
	
	\textbf{Case 2: Incrementing the second parameter ($k=2$)}
	Let $\Theta' = (\theta_1, \theta_2+\delta)$. We replace $\theta_2$ with $\theta_2+\delta$ in the inner integral limits:
	\begin{align*}
	F_{(\theta_1, \theta_2+\delta)}(t) &= \int_{t-\theta_1}^t \left( \int_{t_1-(\theta_2+\delta)}^{t_1} f(t_2) \, dt_2 \right) dt_1\\
	&= \int_{t-\theta_1}^t \left[ \int_{t_1-\theta_2}^{t_1} f(t_2) \, dt_2 + \int_{t_1-\theta_2-\delta}^{t_1-\theta_2} f(t_2) \, dt_2 \right] dt_1\\
	&= \int_{t-\theta_1}^t \int_{t_1-\theta_2}^{t_1} f(t_2) \, dt_2 \, dt_1 + \int_{t-\theta_1}^t \int_{t_1-\theta_2-\delta}^{t_1-\theta_2} f(t_2) \, dt_2 \, dt_1\\
	&=:F_{(\theta_1, \theta_2)}(t)+J_2(t).
	\end{align*}
		We must show $J_2(t) =F_{(\theta_1, \delta)}(t-\theta_2)$.
		Let us compute $F_{(\theta_1, \delta)}(\tau)$ for $\tau = t-\theta_2$:
		\[
		F_{(\theta_1, \delta)}(\tau) = \int_{\tau-\theta_1}^\tau \left( \int_{u-\delta}^u f(t_2) \, dt_2 \right) du.
		\]
		Now, perform the change of variables $u = t_1 - \theta_2$ ($t_1=u+\theta_2$). This gives
		\[
		F_{(\theta_1, \delta)}(\tau) = \int_{\tau-\theta_1}^\tau \left( \int_{u-\delta}^u f(t_2) \, dt_2 \right) du
		= 
		\int_{\tau-\theta_1+\theta_2}^{\tau+\theta_2} \left( \int_{t_1-\theta_2-\delta}^{t_1-\theta_2} f(t_2) \, dt_2 \right)
		dt_1.
		\]
		Since $\tau=t-\theta_2$, we get
		\[
		F_{(\theta_1, \delta)}(t-\theta_2) 
		=	\int_{t-\theta_1}^{t} \left( \int_{t_1-\theta_2-\delta}^{t_1-\theta_2} f(t_2) \, dt_2 \right)
		dt_1=J_2(t),
		\]
		as needed. 
	Thus, when we increment the second paramter, the formula $F_{(\theta_1, \theta_2+\delta)}(t) = F_{(\theta_1, \theta_2)}(t) + F_{(\theta_1, \delta)}(t-\theta_2)$ still holds, and the decomposition is valid irrespective of which parameter is incremented.
\end{proof}

\begin{proof}[Proof of Theorem~\ref{thm:uniformity}]		
	
	\textbf{Step 1: The Baire Category Theorem}
	Fix an arbitrary $\epsilon > 0$. We wish to find a time $T$ such that $|F_\Theta(t)| < \epsilon$ for all $t > T$ and all $\Theta \in S$.
	For each integer $m \ge 1$, define the set:
	\[ E_m = \left\{ \Theta \in S : \sup_{t \ge m} |F_\Theta(t)| \le \epsilon \right\}. \]
	Since $F_\Theta(t)$ is a continuous function of $\Theta$, the sets $E_m$ are closed.
	By assumption, for every fixed $\Theta$, $F_\Theta(t) \to 0$. Thus, every $\Theta$ belongs to some $E_m$ for $m$ large enough.
	\[ S = \bigcup_{m=1}^\infty E_m. \]
	By the Baire Category Theorem, at least one $E_{m_0}$ has a non-empty interior. Thus, there exists a closed ``box'' $B \subset E_{m_0}$ defined by:
	\[ B = \prod_{i=1}^2 [\alpha_i, \beta_i] \subset [0,1]^2, \quad \text{with } \beta_i - \alpha_i > 0. \]
	Let $\delta = \min_i (\beta_i - \alpha_i)$.
	For any $\Psi \in B$ and any $t \ge m_0$, we have $|F_\Psi(t)| \le \epsilon$.
	
	\textit{Step 2: Shifting to the Origin (Rigorous Induction)}
	We now show that this bound on the box $B$ implies a bound on the small box near the origin, $S_\delta = [0, \delta]^2$.
	Let $\Phi = (\phi_1,\phi_2) \in S_\delta$. We want to bound $|F_\Phi(t)|$.
	We proceed by finite induction on the number of "small" components. Let $V_k$ be the set of parameter vectors with exactly $k$ components from the "Small" set $[0, \delta]$ and $2-k$ components from the "Large" set $[\alpha_i, \beta_i]$.
	
	\begin{itemize}
		\item \textbf{Hypothesis $H_k$:} There exists a time horizon $T_k$ and a constant $C_k$ such that for any $\Psi \in V_k$, $|F_\Psi(t)| \le C_k \epsilon$ for all $t \ge T_k$.
		
		\item \textbf{Base Case ($k=0$):}
		$\Psi \in V_0$ means all components are in $[\alpha_i, \beta_i]$. Thus $\Psi \in B$.
		By Step 1, for $t \ge m_0$, $|F_\Psi(t)| \le \epsilon$.
		So, $H_0$ holds with $T_0 = m_0$ and $C_0 = 1$.
		
		\item \textbf{Inductive Step (From $k$ to $k+1$):}
		Assume $H_k$ holds. Consider a vector $\Psi_{next} \in V_{k+1}$.
		Assume the $(k+1)$-th small component is at index $j$, specifically $\phi_j$.
		We relate $\Psi_{next}$ to vectors in $V_k$ using the Decomposition Lemma.
		
		Consider two specific vectors in $V_k$:
		1.  $\Psi_{large}$: Identical to $\Psi_{next}$, but replace the component $\phi_j$ with $\alpha_j + \phi_j$.
		Note: Since $0 \le \phi_j \le \delta \le \beta_j - \alpha_j$, we have $\alpha_j \le \alpha_j + \phi_j \le \beta_j$. Thus, this component is now "Large" (in $B$). This vector has one fewer "Small" component ($k$ total).
		2.  $\Psi_{base}$: Identical to $\Psi_{next}$, but replace the component $\phi_j$ with $\alpha_j$. This component is explicitly "Large". This vector also has $k$ "Small" components.
		
		Apply the Decomposition Lemma to the vector $\Psi_{base}$ with increment $\phi_j$:
		\[ F_{\Psi_{large}}(t) = F_{\Psi_{base}}(t) + F_{\Psi_{next}}(t - \alpha_j). \]
		(Here, $\Psi_{large} = \Psi_{base} + \phi_j e_j$, and the "shifted" term corresponds to $\Psi_{next}$ because the parameter at index $j$ becomes $\phi_j$).
		
		Rearranging for the term we want to bound:
		\[ F_{\Psi_{next}}(t - \alpha_j) = F_{\Psi_{large}}(t) - F_{\Psi_{base}}(t). \]
		Let $\tau = t - \alpha_j$. Then $t = \tau + \alpha_j$.
		\[ |F_{\Psi_{next}}(\tau)| \le |F_{\Psi_{large}}(\tau + \alpha_j)| + |F_{\Psi_{base}}(\tau + \alpha_j)|. \]
		
		If $\tau \ge T_k$, then $\tau + \alpha_j \ge T_k$.
		Since $\Psi_{large}, \Psi_{base} \in V_k$, by the inductive hypothesis $H_k$:
		\[ |F_{\Psi_{large}}(\tau + \alpha_j)| \le C_k \epsilon \quad \text{and} \quad |F_{\Psi_{base}}(\tau + \alpha_j)| \le C_k \epsilon. \]
		Thus:
		\[ |F_{\Psi_{next}}(\tau)| \le 2 C_k \epsilon. \]
		So $H_{k+1}$ holds with $C_{k+1} = 2 C_k$ and $T_{k+1} = T_k$ (assuming $T_k$ is sufficiently large to satisfy decomposition validity).
	\end{itemize}
	
	After $2$ steps, we reach $V_2$, which is exactly $S_\delta$. Thus, for all $\Phi \in S_\delta$, and for $t$ sufficiently large:
	\[ |F_\Phi(t)| \le 4 \epsilon. \]
	
	\textit{Step 3: Extension to all of $S$}
	Any $\Theta \in S$ is a finite sum of elements from $S_\delta$: $\Theta = \sum_{j=1}^2 \Psi_j e_j$ with $\Psi_j \in [0, \delta]^2$ and $N = \lceil 1/\delta \rceil$.
	Repeated application of the Decomposition Lemma allows us to write:
	\[
	F_\Theta(t) = \sum_{j=1}^N F_{\Psi_j}(t - \tau_j),
	\]
	where $\tau_j$ are accumulated shifts. Since each $\Psi_j \in S_\delta$, we have $|F_{\Psi_j}(t)| \le 4 \epsilon$ for sufficiently large $t$.
	Thus, for $t$ large enough:
	\[ |F_\Theta(t)| \le \sum_{j=1}^N |F_{\Psi_j}(t - \tau_j)| \le 4N  \epsilon. \]
	Since $\epsilon$ was arbitrary, $\lim_{t \to \infty} \sup_{\Theta \in S} |F_\Theta(t)| = 0$.
\end{proof}

\section{Stability with Unbounded and Highly Oscillatory Forcing}
We now give an example, showing how badly behaved $f$ (oscillating with unbounded amplitude, and with unboundedly high frequency)
can still lead to asymptotically stable second order systems, in which the state converges, the velocity is bounded, but may have high frequency, while the acceleration is unbounded. 
\begin{lemma} \label{lemma.osc}			Let $A: [0, \infty) \to \mathbb{R}$ be a strictly positive, strictly increasing function of class $C^1$ such that $\lim_{t \to \infty} A(t) = \infty$. 
	Define 
	\begin{equation}  \label{example:oscf}
		f(t) = A(t)\sin\left(\int_0^t A(s) \, ds\right), \quad t\geq 0.
	\end{equation} 
	Then $\limsup_{t\to\infty} |f(t)|=+\infty$. If $y_1$ and $y_2$ are defined by \eqref{eq.y1} and \eqref{eq.y2}, then 
	\begin{enumerate}
		\item[(i)] $y_1(t)$ is bounded, but does not tend to zero as $t \to \infty$.
		\item[(ii)] $\lim_{t \to \infty} y_2(t) = 0$.
	\end{enumerate}
\end{lemma}
We prove this result in a moment. Immediately, however, it can be used to characterise the asymptotic behaviour of any second order linear equation forced by $f$. 
\begin{proposition}
	Let $a>0$, $b>0$ and suppose $f$ is given by \eqref{example:oscf} where $A$ obeys the above properties. Then the solution of \eqref{eq.x} obeys $x''(t)=f(t)+O(1)$ as $ t\to\infty$, and 
	\[
	\limsup_{t\to\infty}|x''(t)|=+\infty, \quad \limsup_{t\to\infty}|x'(t)|\in (0,\infty), \quad \lim_{t\to\infty}x(t)=0. 
	\]
\end{proposition}
\begin{proof}
	Since $y_2(t)\to 0$ as $t\to\infty$, by Theorem~\ref{thm.xf}, we have that $x(t)\to 0$ as $t\to\infty$. Since $y_1$ is bounded, we have that $x'$ and $x$ are bounded. Suppose by way of contradiction, that $x'(t)\to 0$ as $t\to\infty$. Since we already know that $x(t)\to 0$ as $t\to\infty$, by Theorem \ref{thm.xxprintf} it follows that $y_1(t)\to 0$ as $t\to\infty$, which is a contradiction. Hence $x'$ is bounded, but does not tend to zero. Since $x$ and $x'$ are bounded, and $f$ is unbounded, from $x''=-ax'-bx+f$, we see that $x''$ is unbounded, with the claimed representation.
\end{proof}

\begin{proof}[Proof of Lemma~\ref{lemma.osc}]
	First, let $B(t) = \int_0^t A(s) \, ds$. Since $A(t) > 0$, $B(t)$ is strictly increasing. Furthermore, since $A$ is increasing and positive, $A(t) \ge A(0) > 0$, implying $B(t) \ge A(0)t$, so $\lim_{t \to \infty} B(t) = \infty$. 
	Since $B'(t) = A(t)$, we have 
	\[
	f(t) = B'(t)\sin(B(t)) = -\frac{d}{dt}\cos(B(t)).
	\]
	\textbf{Proof of (i):} 	We evaluate $y_1(t)$ using integration by parts. 
	This gives
	\begin{align*}
		y_1(t) 
		&= -\cos(B(t)) + e^{-t}\cos(B(0)) + J(t),
	\end{align*}
	where we define $J(t) = \int_0^t e^{-(t-s)}\cos(B(s)) \, ds$. Since $B(0) = 0$, we have $\cos(B(0)) = 1$. Therefore, since the first term on the righthand side is bounded, but does not tend to zero, it is enough to show that $J(t)\to 0$ as $t\to\infty$ to prove that $y_1$ is bounded, but does not tend to zero. 
	
	To do this define $C(t) = \int_0^t \cos(B(s)) \, ds$. Since $B$ is increasing, it has an increasing inverse, so integrating by substituting  $u = B(s)$, 
	we get:
	\[
	C(t) = \int_0^{B(t)} \frac{\cos(u)}{A(B^{-1}(u))} \, du.
	\]
	Next, note $u\mapsto \cos(u)$ has a bounded antiderivative $\sin(u)$. Because $A$ and $B$ are strictly increasing, $u \mapsto A(B^{-1}(u))$ is strictly increasing. Since $\lim_{t \to \infty} A(t) = \infty$, the function $u\mapsto 1/A(B^{-1}(u))$ monotonically decreases to $0$. By Dirichlet's Test for improper integrals, the improper integral converges to a finite limit. Write $L=\lim_{t \to \infty} C(t)$.
	Return to $J(t)$, and apply integration by parts again to get 
	\begin{align*}
		J(t) 
		&= C(t) - \int_0^t e^{-(t-s)} C(s) \, ds,
	\end{align*}
	since $C(0) = 0$. For the integral term, recall that the limit as $t\to\infty$ of the convolution $(j\ast g)(t)$, where $j$ is integrable and  $g$ tends to a limit $L$, is simply $L\int_0^\infty j(s)\,ds$. Since $\exp_1$ has integral 1, the last term has limit $L$. 
	Therefore, $\lim_{t \to \infty} J(t) = L - L = 0$, as required. 
	
	%
	
	\textbf{Proof of (ii):} We evaluate $y_2(t)$ using integration by parts. This gives 
	\begin{align*}
		y_2(t)
		&= t e^{-t} + \int_0^t (t-s)e^{-(t-s)}\cos(B(s)) \, ds - \int_0^t e^{-(t-s)}\cos(B(s)) \, ds 
		= t e^{-t} + K(t) - J(t),
	\end{align*}
	where $K(t) = \int_0^t (t-s)e^{-(t-s)}\cos(B(s)) \, ds$. We already established that $J(t) \to 0$ as $t\to\infty$ and trivially $t e^{-t} \to 0$ as $t\to\infty$. We need only show that $K(t) \to 0$ as $t\to\infty$. To do this, rewrite $K(t)$ using $C'(s) = \cos(B(s))$ and apply integration by parts to get 
	\begin{align*}
		K(t) 
		&= - \int_0^t (t-s)e^{-(t-s)}C(s) \, ds + \int_0^t e^{-(t-s)}C(s) \, ds.
	\end{align*}
	From the proof of (i), we know the second integral $\int_0^t e^{-(t-s)}C(s) \, ds \to L$ as $t\to\infty$. For the first integral, observe that the function $j(t)=te^{-t}$ has integral 1, so applying the principle above again, we have 
	\[
	\int_0^t (t-s) e^{-(t-s)} C(s)\,ds \to L\int_0^\infty s e^{-s}\,ds = L, \quad t\to\infty.
	\]
	Consequently, $\lim_{t \to \infty} K(t) = -L + L = 0$, from which we have $y_2(t)\to 0$, as required. 
\end{proof}

				\bibliographystyle{unsrt}

\begin{thebibliography}{10}
	
	\bibitem{AL:2023(AppliedNumMath)}
	J.~A.~D. Appleby and E.~Lawless.
	\newblock Mean square asymptotic stability characterisation of perturbed linear stochastic functional differential equations.
	\newblock {\em Applied Numerical Mathematics}, 2023.
	
	\bibitem{AL:2023(AppliedMathLetters)}
	J.~A.~D. Appleby and E.~Lawless.
	\newblock Solution space characterisation of perturbed linear volterra integrodifferential convolution equations: The ${L}^p$ case.
	\newblock {\em Applied Mathematics Letters}, 146:108825, 2023.
	
	\bibitem{AL:2024}
	J. A. D. Appleby and E. Lawless, Weighted L$^\infty$ Asymptotic Characterisation of Perturbed
	Autonomous Linear Ordinary and Stochastic Differential Equations: Part I - ODEs, 2024,
	arXiv, 40pp.
	
	\bibitem{AL:panto24}
	J. A. D. Appleby and E. Lawless, Characterisation of asymptotic behaviour of perturbed deterministic and stochastic pantograph equations, 2025,
	arXiv, 42pp.
	
	\bibitem{GLS}
	G.~Gripenberg, S.-O. Londen, and O.~Staffans.
	\newblock {\em Volterra Integral and Functional Equations}.
	\newblock Encyclopedia of Mathematics and it's Applications. Cambridge University Press, 1990.
	
	\bibitem{SY67a}
	A.~Strauss and J.~A. Yorke.
	\newblock Perturbation theorems for ordinary differential equations.
	\newblock {\em Journal of Differential Equations}, 3:15--30, 1967.
	
	\bibitem{SY67b}
	A.~Strauss and J.~A. Yorke.
	\newblock On asymptotically autonomous differential equations.
	\newblock {\em Mathematical Systems Theory}, 1:175--182, 1967.
\end{thebibliography}

\end{document}